\documentclass[11pt,letterpaper]{amsart}
\usepackage{amssymb}
\usepackage[dvips]{color}
\usepackage{amsmath}
\usepackage{amsthm}
\usepackage{bbm}
\usepackage[initials]{amsrefs}

\usepackage{lineno}

\allowdisplaybreaks[4]

 \newtheorem{thm}{Theorem}[section]
 
 \newtheorem{lem}[thm]{Lemma}
 \newtheorem{prop}[thm]{Proposition}
 \theoremstyle{definition}
 
 \theoremstyle{remark}
 \newtheorem{rem}[thm]{Remark}
 \numberwithin{equation}{section}

 \newcommand{\A}{\mathcal{A}}

 \newcommand{\C}{\mathbb{C}}

 \newcommand{\abs}[1]{\left\vert#1\right\vert}

 \newcommand{\norm}[1]{\left\Vert#1\right\Vert}

\begin{document}

\title[Laplacian masa]{Measure-Multiplicity of the Laplacian masa}

\author[Dykema]{Ken Dykema$^{*}$}

\address{\hskip-\parindent
Ken Dykema, Department of Mathematics, Texas A{\&}M University,
College Station TX 77843-3368, U.S.A.}
\email{kdykema@math.tamu.edu}
\thanks{\footnotesize ${}^{*}$Research supported in part by NSF grant DMS--0901220.}

\author[Mukherjee]{Kunal Mukherjee}
\address{\hskip-\parindent
Kunal Mukherjee, Institute of Mathematical Sciences
C.I.T Campus, Taramani \\
Chennai - 600113\\
India}

\email{kunal@imsc.res.in}

\begin{abstract}
It is shown that for the Laplacian masa in the free group factors, the orthocomplement of the associated
Jones' projection is an infinite direct sum of  coarse bimodules.
\end{abstract}



\maketitle


\section{Introduction}

Let $2\le N<\infty$ and consider the free group factor $L(\mathbb{F}_{N})$, which is
the group von Neumann algebra of the nonabelian free group $\mathbb{F}_{N}$ on $N$ generators,
and is a type II$_1$ factor.
We denote by $a_i$, $(1\le i\le N)$ the canonical unitary generators of $L(\mathbb{F}_{N})$ corresponding
to the free generators of the free group.
A maximal abelian subalgebra (or masa) of a II$_1$--factor is a self--adjoint, commutative subalgebra
that is maximal with respect to these properties.
Various kinds of masas were introduced by Dixmier~\cite{Dix} and have been extensively studied.
See, for example, the book~\cite{SS}.
A so--called generator masas in $L(\mathbb{F}_N)$ is a subalgebra generated by one of the $a_i$.
There are $N$ of them, and they are
easily seen
to be singular (see \cite{Dix} for the definition), as the cyclic subgroup
generated by  the corresponding group element is malnormal in $\mathbb{F}_{N}$ (see~\cite{JS}).
In fact, more is true.
The Puk\'{a}nszky invariant (see \cite{Puk}) is $\{\infty\}$ \cite{RS},
the malnormality above forces the single generator masas to be strongly mixing
\cite{JS} and the orthocomplement of the associated
Jones' projection is a coarse bimodule over the masa \cite[Cor.~4.6]{Muk3}.
Note that single generator masas are conjugate to each other by outer automorphisms
and in fact they are not inner conjugate, as they are pairwise orthogonal
in the sense of Popa \cite{Pop1}, \cite{Pop2}.

In \cite{Muk2} uncountably many pairwise non-conjugate singular masas
in the free group factors each with fixed Puk\'{a}nszky invariant were exhibited.
These masas were
obtained by manipulating suitable masas in the hyperfinite $\rm{II}_{1}$
factor.

The radial or the Laplacian masa in $L(\mathbb{F}_{N})$ is the
von Neumann subalgebra generated by $\sum_{i=1}^{N}(a_{i}+a_{i}^{-1})$.
This subalgebra has been studied extensively for computing the spectra
of convolutors and for understanding the harmonic analysis and
representation theory of $\mathbb{F}_{N}$. That this subalgebra is a masa
was initially proved by Pytlik \cite{Pyt}.
See \cite{SS} for an
operator-algebraic proof. That the radial masa is singular was proved
by R\u{a}dulescu \cite{Rad}, and, interestingly the
singularity was proved by the calculation of it's Puk\'{a}nszky invariant, which is also $\{\infty\}$.
Several attempts have been made to decide the conjugacy (by an automorphism)
of the radial and the single generator masas.
However, all conjugacy invariants that have been computed for both the radial and the single
generator masas have agreed.
For example, they have the same  Puk\'{a}nszky invariant, each is maximal injective \cite{CFRW}, \cite{Pop4}
and each is strongly mixing \cite[an argument analogous to
Theorem 4.1]{Muk3}, \cite{JS}, \cite{SS}.
While it is easy to see that the single
generator and the radial masas are not inner conjugate, the problem of
deciding the conjugacy as stated above remains open.
It is worth noting that there is no natural candidate of a `radial masa'
in $L(\mathbb{F}_{\infty})$.

The measure--multiplicity invariant of a masa $A\subseteq M$ was introduced in~\cite{NS}
and named in~\cite{DSS}, though in essence it
has been known for a long time;  it has been studied recently in  \cite{Muk1} and \cite{Muk2}.
It is a way of describing the $A,A$--bimoduled decomposition of $L^2(M)\ominus L^2(A)$.
If the masa $A$ has a separable, unital and weakly dense C$^*$--subalgebra isomorphic to $C(Y)$
then the measure--multiplicity invariant of $A$ consists (up to equivalence) of a pair $(\mu,m)$,
where $\mu$ is a measure on $Y\times Y$ and a measurable function $m:Y\times Y\to\{1,2,\ldots,\infty\}$
whose essential range equals the Puk\'{a}nszky invariant.
We call $\mu$ the left--right measure of $A\subseteq M$, and $m$ the multiplicity function.
In this paper we calculate the
left-right measure of the Laplacian masa $A\subseteq M=L(\mathbb{F}_N)$
and show that it is Radon--Nikodym equivalent to product
measure $\lambda\otimes\lambda$, where $\lambda$ is the measure on $Y$ arising from the trace of $M$
restricted to $A$.
This shows that the $A,A$--bimodule $L^2(M)\ominus L^2(A)$ is isomorphic to the
direct sum of infintely many copies of the coarse bimodule $L^2(A)\otimes L^2(A)$.

That the left-right measure of the Laplacian masa is absolutely continous
with respect to the product measure follows (but not obviously) from results in \cite{SS}.
However, had the left-right measure been absolutely continous with respect
to the product measure and not equivalent to it, the conjugacy
of the single generator and the Laplacian masas would have been settled
upon consideration of the measure-multiplicity invariant and
would have implied that the Laplacian masa has no free complement
\cite{DSS}.
The left-right measure of every masa in the free group
factors must contain a portion of the product measure as a summand.
This is a deep theorem of Voiculescu~\cite{Voi}.
However, for all known examples of masas in free group factors, the
associated bimodule contains a copy of the coarse bimodule as a direct summand.
It is worth noting that there are no
known examples of masas in $\rm{II}_{1}$ factors for which the left-right
measure is absolutely continuous with respect to product class and not
equivalent to it; a problem which has connections
to open questions in spectral realizations of dynamical systems as well;
see \cite{Muk3}, \cite{Muk2}.
Our proof that the left-right measure
of the Laplacian masa is indeed the class of product measure relies on the calculations and results
of R\u{a}dulescu~\cite{Rad}.

In~\S2 we recall notations
and formulae from \cite{Rad}. In \S3 we will compute the
left-right measure.

\section{R\u{a}dulescu's notation and formulae}

Here are some notations and facts from \cite{Rad} that will be used
in \S3.
The GNS Hilbert space of $L(\mathbb{F}_{N})$ associated to its
trace $\tau$ is $\ell^{2}(\mathbb{F}_{N})$. The associated inner product is denoted by
$\langle\cdot,\cdot\rangle_{2}$.
Let $\C[\mathbb{F}_{N}]$ denote
the group algebra of $\mathbb{F}_{N}$, i.e., the collection of
finite sums of the form $\sum_{w\in \mathbb{F}_{N}}\lambda_{w}w$,
$\lambda_{w}\in \C$ equipped with usual product structure. The void word
$\emptyset$ corresponds to the identity of $L(\mathbb{F}_{N})$.  Let
$\abs{\cdot}$ denote the word length function on $\mathbb{F}_{n}$.
Write
\begin{align}
\nonumber \chi_{n}=\sum_{\abs{w}=n}w, \text{ }n\geq 0.
\end{align}
Thus $\chi_{0}=1$, $\chi_1$  generates the Laplacian masa  and the following recurrence relations hold.
\begin{equation}\label{recur}
\begin{aligned}
\chi_{1}\chi_{1}&=\chi_{2} +2N, \\
\chi_{1}\chi_{n}&=\chi_{n}\chi_{1}=\chi_{n+1}+(2N-1)\chi_{n-1},\quad(n\geq 2).
\end{aligned}
\end{equation}

Let $A=W^{*}(\sum_{i=1}^{N}(a_{i}+a_{i}^{-1}))\subset L(\mathbb{F}_{N})$
denote the Laplacian masa. Then from equation~\eqref{recur} it follows that $A$ is the closure in $w.o.t.$
of $\text{span} \{\chi_{n}:n\geq 0\}$. Moreover, $\{\chi_{n}:n\geq 0\}$ forms
a complete set of orthogonal vectors in $L^{2}(A)$. Moreover, from equation
\eqref{recur} it follows that for all $n\geq 0$,
$\chi_{n}$ is a polynomial in $\chi_{1}$.

For any $l\geq 0$, denote by $M_{0}^{l}$  the span of words
of length $l$ in $\mathbb{F}_{N}$. Let $q_{l}$ denote the
projection of $\ell^{2}(\mathbb{F}_{N})$
onto $M_{0}^{l}$, and, let $S_{l}\subseteq M_{0}^{l}$ be the subspace
spanned by
$\{q_{l}(\chi_{1}w),q_{l}(w\chi_{1}): w\in \mathbb{F}_{N}, \abs{w}\leq l-1\}$.
Thus $S_{0}=\{0\}$. For a vector $\xi\in M_{0}^{l}$, $l\geq 1$, and integers
$r,s$ write
\begin{equation}
\xi_{r,s}=\begin{cases}
           q_{r+s+1}(\chi_{r}\xi \chi_{s}) &\text{ if }r,s\geq 0,\\
           0 &\text{ otherwise.}
          \end{cases}
\end{equation}
The next statement
is parts of Lemma 2 and 3 in \cite{Rad}. We state it here for convenience.

\begin{lem}\label{expanding_sums}
Let $\varepsilon,\varepsilon^{\prime}\in \{\pm 1\}$
and let $\beta,\beta^{\prime}\in M_{0}^{1}\ominus S_{1}$ be such that
\begin{align}
\nonumber &\beta =\sum_{\abs{w}=1}c_{w}w, \text{ with }c_{w}=\varepsilon c_{w^{-1}}, \\
\nonumber &\beta^{\prime} =\sum_{\abs{w}=1}d_{w}w, \text{ with }d_{w}=\varepsilon^{\prime} d_{w^{-1}}.
\end{align}
Then for any $n,n^{\prime},m,m^{\prime}\geq 0$,
\begin{align}
\nonumber &\langle \beta_{n,m},\beta'_{n^{\prime},m^{\prime}}\rangle_{2}
=\delta_{\varepsilon,\varepsilon^{\prime}}\delta_{n+m,n^{\prime}+m^{\prime}}(2N-1)^{n+m}
(-\varepsilon(2N-1))^{-\abs{n-n^{\prime}}}\langle \beta,\beta^{\prime}\rangle_{2}.
\end{align}
Moreover,
\begin{align}
\nonumber \chi_{n}\beta\chi_{m}=&\beta_{n,m}-(\beta_{n-2,m}+\beta_{n,m-2}+\varepsilon \beta_{n-1,m-1})\\
\nonumber &+ \sum_{k\geq 2}(-\varepsilon)^{k}(\varepsilon \beta_{n-k-1,m-k+1}+ \varepsilon \beta_{n-k+1,m-k-1}+2\beta_{n-k,m-k}).
\end{align}
\end{lem}

\begin{lem}\label{lem:betaip}
For $\beta$ as in the statement of the previous lemma and $n,m\geq 0$, we have
\begin{equation}\label{eq:cbccases}
\langle \chi_{n}\beta\chi_{m},\beta\rangle_{2}\\
=\begin{cases}
 (-\varepsilon)^{\frac{n+m}{2}}\varepsilon \norm{\beta}_{2}^{2},& n+m>2,\,|n-m|=2 \\
2(-\varepsilon)^{n}\norm{\beta}_{2}^{2},& n+m>2,\,n=m \\
-\norm{\beta}_{2}^{2},&n+m=2,\,(n,m)\neq(1,1) \\
-\varepsilon\norm{\beta}_{2}^{2},&(n,m)=(1,1) \\
\norm{\beta}_{2}^{2},&(n,m)=(0,0) \\
0,&\text{otherwise.}
\end{cases}
\end{equation}
\end{lem}

\begin{proof}
{}From Lemma \ref{expanding_sums} we have
\begin{align}
\langle \chi_{n}\beta\chi_{m},\beta\rangle_{2}
 &=\langle \beta_{n,m},\beta\rangle_{2} \notag \\
 &\;-\biggl(\langle\beta_{n-2,m},\beta\rangle_{2}
 +\langle\beta_{n,m-2},\beta\rangle_{2}
 +\varepsilon \langle\beta_{n-1,m-1},\beta\rangle_{2}\biggr) \label{innerprod1} \\
&\;+\sum_{k\geq 2}(-\varepsilon)^{k}\begin{aligned}[t]\biggl(&\varepsilon\langle \beta_{n-k-1,m-k+1},\beta\rangle_{2}
 +\varepsilon\langle\beta_{n-k+1,m-k-1},\beta\rangle_{2} \\
 &+2\langle \beta_{n-k,m-k},\beta\rangle_{2}\biggr).\end{aligned} \label{innerprod2}
\end{align}
Note that by definition $\beta=\beta_{0,0}$.
Using Lemma \ref{expanding_sums} again, we see that
the only term in the sum~\eqref{innerprod2} that can be nonzero are when $k=\frac{n+m}{2}$;
moreover, the expression~\eqref{innerprod1} is nonzero only when $n+m=2$
while $\langle \beta_{n,m},\beta\rangle_{2} $ is nonzero only when $(n,m)=(0,0)$.
{}From these facts, the formula~\eqref{eq:cbccases} follows easily.
\end{proof}

\section{Left-Right Measure}

Let $M$ be a separable
$\rm{II}_{1}$ factor equipped with its faithful normal tracial state $\tau$.
This trace induces a Hilbert space norm on $M$, given by
$\norm{x}_{2}=\tau(x^{*}x)^{\frac{1}{2}}$, $x\in M$. The
Hilbert space completion of $M$ with respect to
$\norm{\cdot}_{2}$ is denoted by $L^{2}(M)$. The associated inner product on
$L^{2}(M)$ will be denoted by $\langle\cdot,\cdot\rangle_{2}$. Let
$M$ act on $L^{2}(M)$ via left multiplication.
Let $J$ denote the Tomita's modular operator on
$L^{2}(M)$, obtained by extending the densely defined
map $J:M\mapsto M$ given by $Jx=x^{*}$.
Let $A\subset M$ be a masa.
Let $e_{A}$ be the Jones projection associated to $A$. Denote
$\A=A\vee JAJ$. It is known that $e_{A}\in
\A$.

Choose a compact Hausdorff space $X$ such that $C(X)\subset A$ is a
norm separable unital $C^{*}$-subalgebra which is \emph{w.o.t.} dense in $A$. The trace $\tau$ restricted to $C(X)$ gives rise to a
probability measure $\nu$ on $Y$. We complete $\nu$ if necessary, so that
$A$ is isomorphic to $L^{\infty}(Y,\nu)$.

For definitions and details about measure-multiplicity invariant
and left-right measure see \cite{DSS} and \cite{Muk1}.
For $\zeta\in L^{2}(M)$, let
$\kappa_{\zeta}:C(X)\otimes C(X)\mapsto \C$ be the linear functional defined by,
\begin{align}
\nonumber \kappa_{\zeta}(a\otimes b)=\langle
a\zeta b,\zeta\rangle,\text{ } a,b\in C(X).
\end{align}
Then $\kappa_{\zeta}$ induces an unique complex Radon
measure $\eta_{\zeta}$ on $X\times X$ given by,
\begin{equation}\label{eq:eta}
\kappa_{\zeta}(a\otimes b)=\int_{X\times X}a(t)b(s)d\eta_{\zeta}(t,s).
\end{equation}
There is a vector $0\neq \zeta_{0}\in L^{2}(M)\ominus L^{2}(A)$ such that
$[\eta_{\zeta_{0}}]$ is the left-right measure of $A$.

\smallskip

We will use the result of the following elementary calculation, which we show for convenience.

\begin{lem}\label{lem:trig}
Let $x$ be a real or complex number of modulus $|x|<1$, let $\theta$, $\phi$ and $r$
be real numbers.
Then
\begin{multline}\label{eq:trig}
\sum_{n=0}^\infty x^n\sin n\theta\sin(n+r)\phi=
\frac12\left(\frac{\cos(r\phi)-x\cos(\theta+(r-1)\phi)}{1-2x\cos(\theta-\phi)+x^2}\right) \\
 -\frac12\left(\frac{\cos(r\phi)-x\cos(\theta-(r-1)\phi)}{1-2x\cos(\theta+\phi)+x^2}\right).
 \end{multline}
 \end{lem}
\begin{proof}
The left--hand--side of~\eqref{eq:trig} equals
\[
\begin{aligned}
&-\frac14\sum_{n=0}^\infty x^n(e^{in\theta}-e^{-in\theta})(e^{i(n+r)\phi}-e^{-i(n+r)\phi}) \\
&\quad=-\frac14\sum_{n=0}^\infty x^n(e^{in(\theta+\phi)+ir\phi}-e^{-in(\theta-\phi)+ir\phi}
  -e^{in(\theta-\phi)-ir\phi}+e^{-in(\theta+\phi)-ir\phi}) \\
&\quad=-\frac14\biggl(
 \frac{e^{ir\phi}}{1-xe^{i(\theta+\phi)}}
 -\frac{e^{ir\phi}}{1-xe^{-i(\theta-\phi)}}
 -\frac{e^{-ir\phi}}{1-xe^{i(\theta-\phi)}}
 +\frac{e^{-ir\phi}}{1-xe^{-i(\theta+\phi)}}\biggr).
\end{aligned}
\]
Now putting the first and fourth terms as well as the second and third terms over a common denominator
finishes the calculation.
\end{proof}

\begin{rem}\label{rem:lambda}
The Laplacian masa is generated by the self-adjoint operator
$\chi_{1}$.
The computation of the norm of $\chi_1$ and the generating series of its moments
goes back to Kesten~\cite{K}.
{}From this, using Stieljes inversion, the distribution measure of $\chi_1$ can be found, and
one obtains that the spectrum of $\chi_1$ is
$\sigma_{\chi_{1}}=[-2\sqrt{2N-1},2\sqrt{2N-1}]$
and the distribution is Lebesgue absolutely continuous.
Let us write
$a_{N}=2\sqrt{2N-1}$.
So the
weakly dense separable $C^{*}$--subalgebra
$C^{*}(\{1,\chi_{1}\})$
of the Laplacian masa $A$
is identified with $C[-a_{N},a_{N}]$ in such a way that $\chi_1$ is identified with
the function $t\mapsto t$ on $[-a_N,a_N]$.
Thus,
from the recurrence relations~\ref{recur}, $\chi_{n}$ is identified with a polynomial of degree $n$
in $C[-a_{N},a_{N}]$.
Also note that
$\chi_{n}$, $n=0,1,\cdots$, is total family of orthogonal vectors (with respect to
$\langle\cdot,\cdot\rangle_{2}$) in $L^{2}(A)$; thus
$\frac{\chi_{n}}{\norm{\chi_{n}}_{2}}$, $n=0,1,\cdots$, is an orthonormal
basis of $L^{2}(A)$. Finally, the trace $\tau$ restricted to $A$ is identified
with a probability measure $\lambda$ on $[-a_{N},a_{N}]$ for which $(\chi_{n})_{n\ge0}$
is a family of orthogonal polynomials. Note that $\lambda$ is Radon--Nikodym equivalent to
the Lebesgue measure on $[-a_{N},a_{N}]$.
It is straightforward to check from the definition that
$\norm{\chi_{n}}_{2}=\sqrt{2N}(2N-1)^{\frac{n-1}{2}}$ for all $n\geq 1$.
\end{rem}

Consider the function $f:[-a_{N},a_{N}]\times [-a_{N},a_{N}]\mapsto \mathbb{R}$ defined by
\begin{equation}\label{Radondvt}
f(t,s)=
\begin{aligned}[t]
&\biggl(1+\frac{\chi_{1}(t)\chi_{1}(s)}{\norm{\chi_{1}}_{2}^{4}}
-\frac{\chi_{2}(t)+\chi_{2}(s)}{\norm{\chi_{2}}_{2}^{2}}\biggr)
\\
&+\sum_{n=2}^\infty\biggl(2\,\frac{\chi_n(t)\chi_n(s)}{\norm{\chi_n}_2^4}
-\frac{\chi_{n-1}(t)\chi_{n+1}(s)+\chi_{n+1}(t)\chi_{n-1}(s)}{\norm{\chi_{n-1}}_2^2\,\norm{\chi_{n+1}}_2^2}\biggr)
\end{aligned}
\end{equation}

It is possible to give a closed form of $f$ in terms of
transcendental functions. However, for our purpose it is
more important to know the `size' of the zero set of $f$.

\begin{prop}\label{real_analytic}
The function $f$ defined in equation \eqref{Radondvt} is continuous on $[-a_{N},a_{N}]\times [-a_{N},a_{N}]$
and real analytic in the interior of this set.
\end{prop}

\begin{proof}
Let $t\in [-a_{N},a_{N}]$ and $\theta\in [0,\pi]$.
Consider the homeomorphism $t\mapsto \theta$ from
$[-a_{N},a_{N}]$ onto $[0,\pi]$ such that
\[
\cos\theta=\frac{t}{a_{N}},\quad \sin\theta=\frac{\sqrt{a_{N}^{2}-t^{2}}}{a_{N}}.
\]
Then from p.~1062 of~\cite{Co} we have,  for all $n\geq 1$,
\[
\frac{\chi_{n}(t)}{(2N-1)^{\frac{n}{2}}}=\frac{2\cos\theta \sin n\theta}{\sin\theta}-\frac{2N}{2N-1}\cdot\frac{\sin (n-1)\theta}{\sin\theta}.
\]
Write $b=2N-1$ and $d=\frac{2N-1}{2N}$. Then
\[
\frac{\chi_{n}(t)}{\norm{\chi_{n}}_{2}^{2}}=b^{-n/2}
 \left(\frac{2d\cos\theta \sin n\theta-\sin (n-1)\theta}{\sin\theta}\right).
\]
Since $\frac{|\sin n\theta|}{\sin\theta}\le n$ for all integers $n\ge1$ and all $\theta\in[0,\pi]$
(which is easily shown by induction on $n$), we have
$\frac{|\chi_{n}(t)|}{\norm{\chi_{n}}_{2}^{2}}\le3nb^{-n/2}$.
Thus, the series in equation~\eqref{Radondvt} converges absolutely and uniformly for $(s,t)$ in the domain of $f$.
This implies continuity of $f$ on its domain.

For $t,s\in [-a_{N},a_{N}]$, taking $\phi\in[0,\pi]$ so that
$\cos\phi=\frac{s}{a_{N}}$ and $\sin\phi=\frac{\sqrt{a_{N}^{2}-s^{2}}}{a_{N}}$, we have
\[
\begin{aligned}
\frac{\chi_{n}(t)\chi_{n}(s)}{\norm{\chi_{n}}_{2}^4}
&=b^{-n}\left(\frac1{\sin\theta\sin\phi}\right)\cdot \\
&\quad\biggl( \begin{aligned}[t]
  &4d^2\cos\theta\cos\phi \sin n\theta \sin n\phi -2d\cos\theta\sin n\theta\sin(n-1)\phi \\
  &-2d\cos\phi\sin (n-1)\theta\sin n\phi+\sin (n-1)\theta\sin (n-1)\phi\bigg),\end{aligned}
\end{aligned}
\]
and one can write a similar expression for
\[
\frac{\chi_{n-1}(t)\chi_{n+1}(s)+\chi_{n+1}(t)\chi_{n-1}(s)}{\norm{\chi_{n-1}}_{2}^{2}\norm{\chi_{n+1}}_{2}^{2}}.
\]
Now summing over $n$ and using Lemma~\ref{lem:trig}, we obtain a closed form expression for $f(t,s)$
and see that it is a real analytic function of $\theta$ and $\phi$.
Since, for $(t,s)$ in the interior of $[-a_{N},a_{N}]\times [-a_{N},a_{N}]$, $\theta$ and $\phi$ are real analytic
functions of $t$ and $s$, $f(t,s)$ is indeed a real analytic function of $(t,s)$ there.
\end{proof}

\begin{thm}
The left-right measure of $A=W^{*}(\sum_{i=1}^{N}(a_{i}+a_{i}^{-1}))\subset L(\mathbb{F}_{N})$
is the class of product measure.
\end{thm}
\begin{proof}
Let $\beta\in M_{0}^{1}\ominus S_{1}$ be as in Lemma
\ref{expanding_sums} with $\varepsilon=-1$. Let
$\gamma=\frac{\beta}{\norm{\beta}_{2}}$. Let $p_\gamma$ be the orthogonal projection
onto the subspace $\overline{\A\gamma}^{\norm{\cdot}_{2}}$.
We will show below that the measure class of $\eta_\gamma$ on $[-a_N,a_N]\times[-a_N,a_N]$,
which is defined in~\eqref{eq:eta},
is the same as $\lambda\otimes\lambda$, where $\lambda$ is the measure coming from $\tau$
defined in Remark~\ref{rem:lambda}.
This will imply $\lambda\otimes\lambda\ll\nu$, where $\nu$ is the left--right measure of $A$.
Indeed (see sections~5 and~6 of~\cite{DSS}), the restriction of
the left--right action of $A\otimes A$ to the range of $p_\gamma$ is a direct integral of Hilbert spaces
with respect to the measure $\eta_\gamma$ and with multiplicity function constantly one.

By the proof of Theorem~7 of~\cite{Rad}, the central support in $\A^{\prime}$ of $p_\gamma$
is $1-p_1$, and $p_1$ is equal to the Jones projection $e_A$ onto $L^2(A)$.
Moreover, by Lemma~6 of~\cite{Rad}, it follows that $1-p_1$ is greater than or equal to the sum of
infinitely many projections all equivalent in $\A'$ to $p_\gamma$.
It follows that $1-p_1$ is equal to a sum projections, infinitely many of which are equivalent in $\A'$
to $p_\gamma$ and all of which are equivalent in $\A'$ to subprojections of $p_\gamma$.
{}From this, it follows that the left--right measure $\nu$ of $A$ has the same measure class as $\eta_\gamma$,
{\em cf}\/ Proposition~5.8 of~\cite{DSS}.

Now we go about describing $\eta_\gamma$.
For $a,b\in C[-a_{N},a_{N}]$, one has
\begin{align}
\nonumber &a=\sum_{n=0}^{\infty}\langle a,\frac{\chi_{n}}{\norm{\chi_{n}}_{2}}\rangle_{2}\frac{\chi_{n}}{\norm{\chi_{n}}_{2}},\\
\nonumber &b=\sum_{n=0}^{\infty}\langle b,\frac{\chi_{n}}{\norm{\chi_{n}}_{2}}\rangle_{2}\frac{\chi_{n}}{\norm{\chi_{n}}_{2}},
\end{align}
where the series in the above converges in $\norm{\cdot}_{2}$.
Thus using Lemma \ref{lem:betaip} it follows that
\begin{align}
\notag
\langle a\gamma b, \gamma\rangle_{2}
&=\sum_{n,m=0}^{\infty}
\langle a,\chi_{n}\rangle_{2}\,\langle b,\chi_{m}\rangle_{2}
\frac{\langle \chi_{n}\beta \chi_{m},\beta\rangle_{2} }{\norm{\chi_{n}}_{2}^2\,\norm{\chi_{m}}_{2}^2\,\norm{\beta}_{2}^{2}}  \\
&=\tau(a)\tau(b)+\langle a,\chi_1\rangle\,\langle b,\chi_1\rangle
  \,\frac1{\norm{\chi_1}_2^4}  \label{dblsum} \\
&\quad-\biggl(\tau(a)\langle b,\chi_2\rangle
   +\langle a,\chi_2\rangle\tau(b)\biggr)\frac1{\norm{\chi_2}_2^2} \notag \\
&\quad+2\sum_{k=2}^\infty
   \langle a,\chi_k\rangle\,\langle b,\chi_k\rangle
  \,\frac1{\norm{\chi_k}_2^4} \notag \\
&\quad-\sum_{k=1}^\infty\biggl(
    \langle a,\chi_k\rangle\,\langle b,\chi_{k+2}\rangle
   +\langle a,\chi_{k+2}\rangle\,\langle b,\chi_k\rangle\biggr)
    \,\frac1{\norm{\chi_k}_2^2\,\norm{\chi_{k+2}}_2^2} . \notag
\end{align}
Consequently from equation \eqref{Radondvt} and equation \eqref{dblsum} it follows that
\begin{align}
\notag\int_{-a_{N}}^{a_{N}}\int_{-a_{N}}^{a_{N}} a(t)b(s)d\eta_{\gamma}(t,s)=\int_{-a_{N}}^{a_{N}}\int_{-a_{N}}^{a_{N}} a(t)b(s)f(t,s)d\lambda(t)d\lambda(s).
\end{align}
Thus $\eta_{\gamma}\ll \lambda\otimes \lambda$ with
$\frac{d\eta_{\gamma}}{d(\lambda\otimes \lambda)}=f$.

Write $E=\left\{(t,s):f(t,s)=0\right\}$.
The intersection of $E$ with the open set $(-a_N,a_N)\times(-a_N,a_N)$ is the zero set of a real analytic function.
It is easy to see (based on induction on the number of variables) that the Lebesgue measure of the zero
set of a real analytic function in several variables must vanish, unless the function is identically zero.
Thus, we have $(\lambda\otimes \lambda)(E)=0$
and
$\eta_{\gamma}\sim \lambda\otimes \lambda$.
\end{proof}

\end{document}